\documentclass[12pt,reqno]{amsart}

\usepackage{amssymb}
\usepackage{verbatim}
\usepackage{graphicx}
\usepackage[active]{srcltx}

\newcommand{\alp}{\alpha}
\newcommand{\del}{\delta}
\newcommand{\eps}{\varepsilon}
\newcommand{\lam}{\lambda}
\newcommand{\ome}{\omega}
\newcommand{\zet}{\zeta}

\newcommand{\Gam}{\Gamma}
\newcommand{\Del}{\Delta}

\newcommand{\R}{{\mathbb R}}

\newcommand{\cI}{{\mathcal I}}
\newcommand{\cF}{{\mathcal F}}

\newcommand{\prt}{\partial}

\renewcommand{\(}{\left(}
\renewcommand{\)}{\right)}
\newcommand{\lfl}{\left\lfloor}
\newcommand{\rfl}{\right\rfloor}
\newcommand{\lcl}{\left\lceil}
\newcommand{\rcl}{\right\rceil}

\newcommand{\seq}{\subseteq}
\newcommand{\stm}{\setminus}
\newcommand{\est}{\varnothing}

\newcommand{\longc}{,\dotsc,}
\newcommand{\longp}{+\dotsb+}
\newcommand{\longm}{-\dotsb-}
\newcommand{\longe}{=\dotsb=}

\newcommand{\mmod}[1]{\!\!\pmod{#1}}
\newcommand{\bfr}[2]{\(\frac{#1}{2^{#2}}\)}
\newcommand{\tfr}[2]{\(\frac{#1}{3^{#2}}\)}

\newtheorem{lemma}{Lemma}
\newtheorem{theorem}{Theorem}
\newtheorem{corollary}{Corollary}
\newtheorem{proposition}{Proposition}

\theoremstyle{definition}
\newtheorem{example}{Example}

\newcommand{\refc}[1]{\ref{c:#1}}
\newcommand{\reft}[1]{\ref{t:#1}}
\newcommand{\refl}[1]{\ref{l:#1}}
\newcommand{\refp}[1]{\ref{p:#1}}
\newcommand{\refs}[1]{\ref{s:#1}}
\newcommand{\refb}[1]{\cite{b:#1}}
\newcommand{\refe}[1]{\eqref{e:#1}}

\author{Vsevolod F. Lev}
\email{seva@math.haifa.ac.il}
\address{Department of Mathematics, The University of Haifa at Oranim,
         Tivon 36006, Israel}

\title[Edge-isoperimetric problem for Cayley graphs]
      {Edge-isoperimetric problem \\
       for Cayley graphs \\
       and generalized Takagi function}

\subjclass[2010]%
  {Primary: 05C35;    
   secondary: 26A30,  
              26A51,  
              26B25,  
              39B62.} 

\begin{document}
\baselineskip=16pt

\begin{abstract}
Let $G$ be a finite abelian group of exponent $m\ge 2$. For subsets $A,S\seq
G$, denote by $\prt_S(A)$ the number of edges from $A$ to its complement
$G\stm A$ in the directed Cayley graph, induced by $S$ on $G$. We show that
if $S$ generates $G$, and $A$ is non-empty, then
  $$ \textstyle \prt_S(A) \ge \frac{e}m\,|A|\ln\frac{|G|}{|A|}\, . $$
Here the coefficient $e=2.718\ldots$ is best possible and cannot be replaced
with a number larger than $e$.

For homocyclic groups $G$ of exponent $m$, we find an explicit closed-form
expression for $\prt_S(A)$ in the case where $S$ is the ``standard''
generating subset of $G$, and $A$ is an initial segment of $G$ with respect
to the lexicographic order, induced by $S$. Namely, we show that in this
situation
  $$ \prt_S(A) = |G|\,\ome_m(|A|/|G|), $$
where $\ome_2$ is the Takagi function, and $\ome_m$ for $m\ge3$ is an
appropriate generalization thereof. This particular case is of special
interest, since for $m\in\{2,3,4\}$ it is known to yield the smallest
possible value of $\prt_S(A)$, over all sets $A\seq G$ of given size. We give
this classical result a new proof, somewhat different from the standard one.

We also give a new, short proof of the Boros-P\'ales inequality
  $$ \textstyle
     \ome_2\(\frac{x+y}2\) \le \frac{\ome_2(x)+\ome_2(y)}2 + \frac12\,|y-x|, $$
establish an extremal characterization of the Takagi function as the
(pointwise) maximal function, satisfying this inequality and the boundary
condition $\max\{\ome_2(0),$ $\ome_2(1)\}\le 0$, and obtain similar results
for the $3$-adic analog $\ome_3$ of the Takagi function.
\end{abstract}

\maketitle

\section{Introduction: summary of results and background}
The three tightly related objects of study in this paper are the
edge-isoperimetric problem on Cayley graphs, a sequence of Takagi-style
functions, and classes of functions satisfying a certain kind of convexity
condition.

The edge-isoperimetric problem for a graph $\Gam$ on the vertex set $V$ is to
find, for every non-negative integer $n\le |V|$, the smallest possible number
of edges between an $n$-element set of vertices and its complement in $V$.
This classical problem has received much attention in the literature; for the
history, results, variations, and numerous related problems, the reader can
refer to the survey of Bezrukov \refb{be} or the monograph of Harper
\refb{h}.

In the present paper we are concerned with, arguably, the most studied case
where $\Gam$ is a Cayley graph. We use the following notation. Given two
subsets $S,A\seq G$ of a finite abelian group $G$, by $\Gam_S(G)$ we denote
the (directed) Cayley graph, induced by $S$ on $G$, and we write $\prt_S(A)$
for the number of edges in $\Gam_S(G)$ from an element of $A$ to an element
in its complement $G\stm A$; that is,
  $$ \prt_S(A) := |\{ (a,s)\in A\times S\colon a+s\notin A \}|. $$

It is easily seen that if $S$ is symmetric (meaning that $S=-S$, where
$-S:=\{-s\colon s\in S\}$), then $\prt_S(A)$ can be equivalently defined as
the number of edges of the corresponding \emph{undirected} Cayley graph, with
one of the incident vertices in $A$ and another one in $G\stm A$. As a less
trivial fact, we have
  $$ \prt_{-S}(A)=\prt_S(G\stm A)=\prt_S(A); $$
consequently, if $S$ is anti-symmetric (that is, $S\cap(-S)=\est$), then
$\prt_S(A)$ is half the number of edges, joining a vertex from $A$ with a
vertex from $G\stm A$, in the undirected Cayley graph, induced on $G$ by the
set $S\cup(-S)$. We omit detailed explanations since none of these
observations are used below.

Up until now, all the research we are aware of has focused on particular
families of Cayley graphs. In contrast, our first principal result addresses
the general situation.
\begin{theorem}\label{t:main1}
Let $m\ge 2$ be an integer, and suppose that $G$ is a finite abelian group,
the exponent of which divides $m$. Then for any non-empty subset $A\seq G$
and any generating subset $S\seq G$ we have
  $$ \prt_S(A) \ge \frac em\,|A| \ln \frac{|G|}{|A|} $$
(where $e=2.718...$ is Euler's number).
\end{theorem}

The estimate of Theorem~\reft{main1} is sharp in the sense that the
coefficient $e$ cannot be replaced with a larger number.
\begin{example}
For integer $r\ge 1$ and $m\ge 2$, let $G$ be the homocyclic group of
exponent $m$ and rank $r$. Fix arbitrarily a generating subset $S=\{s_1\longc
s_r\}\seq G$ and integer $k\in[1,r]$ and $t\in[1,m-1]$, and consider the set
  $$ A := \{x_1s_1\longp x_rs_r\colon 0\le x_1\longc x_k\le t-1,\,
                                  0\le x_{k+1}\longc x_r\le m-1 \}. $$
Write $\alp:=t/m$. Then $|A|=t^km^{r-k},\,\ln(|G|/|A|)=k\ln(1/\alp)$, and
  $$ \prt_S(A) = kt^{k-1}m^{r-k} = \frac{c(\alp)}{m}\,|A|\,\ln\frac{|G|}{|A|}, $$
where $c(\alp)=\frac{1/\alp}{\ln(1/\alp)}$ can be made arbitrarily close to
$e$ by choosing $t$ and $m$ appropriately.
\end{example}

The proof of Theorem~\reft{main1} and most of other results, presented in the
Introduction, is postponed to subsequent sections.

Below we use the standard notation $C_m^r$ for the homocyclic group of
exponent $m$ and rank $r$. In the case where $m\in\{2,3,4\}$, and $S\seq
C_m^r$ is a generating set of size $|S|=r$, the minimum of $\prt_S(A)$
over all sets $A$ of prescribed size is known to be realized when $A$ is
the set of the lexicographically smallest group elements; this basic fact
due to Harper \refb{h1} (the case $m\in\{2,4\}$) and Lindsey \refb{li}
(the case $m=3$) follows also from our present results, as explained
below. To put the things formally, for a finite, totally ordered set $T$
and a non-negative integer $n\le|T|$, denote by $\cI_n(T)$ the length-$n$
initial segment of $T$; that is, the set of the $n$ smallest elements of
$T$. Consider the group $C_m^r$ along with a fixed generating subset
$S\seq C_m^r$ of size $|S|=r$. We assume that $S$ is totally ordered,
inducing a lexicographic order on $C_m^r$; thus, $\cI_n(C_m^r)$ is the
set of the $n$ lexicographically smallest elements of $C_m^r$. As we have
just mentioned, if $m\in\{2,3,4\}$, then
\begin{equation}\label{e:min-prt}
  \min \{ \prt_S(A) \colon A\seq C_m^r,\ |A|=n \}
                                = \prt_S(\cI_n(C_m^r)),\quad 0\le n\le m^r.
\end{equation}

Surprisingly, to our knowledge, no explicit closed-form expression for the
quantity $\prt_S(\cI_n(C_m^r))$ has ever been obtained. We show that such an
expression can be given in terms of the Takagi function for $m=2$, and an
appropriate $m$-adic version thereof for $m\ge 3$.

For real $x$, let $\|x\|$ denote the distance from $x$ to the nearest
integer. The Takagi function, first introduced by Teiji Takagi in 1903 as an
example of an everywhere
continuous but nowhere differentiable function, is defined by%
  $$ \ome(x) := \sum_{k=0}^\infty 2^{-k} \|2^kx\|. $$
Numerous remarkable properties of this function, applications, and relations
in various fields of mathematics can be found in the recent survey papers by
Lagarias \refb{la} and Allaart-Kawamura \refb{ak}. For the generalization we
need, for real $x$ and $\alp$ let $\|x\|_\alp:=\min\{\|x\|,\alp\}$ (the
distance from $x$ to the nearest integer, truncated at $\alp$), and set
\begin{equation}\label{e:omega-m-def}
  \ome_m(x) := \sum_{k=0}^\infty m^{-k} \|m^kx\|_{1/m},\quad m\ge 2.
\end{equation}
(Thus, $\ome_2$ is just the regular Takagi function.) Since the series in
\refe{omega-m-def} is uniformly convergent, the functions $\ome_m$ are
well-defined and continuous on the whole real line. Furthermore, they are
even functions, periodic with period $1$, vanishing at integers, strictly
positive for non-integer values of the argument, and satisfying
\begin{equation}\label{e:max-omega}
  \max \ome_m \le \sum_{k=0}^\infty m^{-(k+1)} = \frac1{m-1}.
\end{equation}

The reader is invited to compare our second major result against
\refe{min-prt}.
\begin{theorem}\label{t:lex}
For integer $r\ge 1$ and $m\ge 2$, let $S$ be an $r$-element generating
subset of the homocyclic group $C_m^r$. Suppose that an ordering of $S$ is
fixed, inducing a lexicographic ordering of $C_m^r$. Then for any
non-negative integer $n\le m^r$, the set $\cI_n(C_m^r)$ of the $n$
lexicographically smallest elements of $C_m^r$ satisfies
  $$ \prt_S(\cI_n(C_m^r))=m^r\ome_m(n/m^r). $$
\end{theorem}

We remark that for $m\in\{2,3,4\}$, Theorem~\reft{lex} together with
\refe{min-prt} and continuity of $\ome_m$ readily shows that for any fixed
$x\in(0,1)$, if $n_r=(1+o(1))m^rx$ as $r\to\infty$, then
  $$ \min \{ \prt_S(A) \colon A\seq C_m^r,\ |A|=n_r \}
                                               = (1+o(1))\, m^r \ome_m(x). $$
The particular case $m=2$ and $n_r=\lfl 2^rx\rfl$ is the main result of
\refb{g}.

In the Appendix we establish some estimates for the growth rate of the
functions $\ome_m$: specifically, we show that
\begin{align}
  x\log_2(1/x) \le\, & \ome_2(x) \le x\log_2(4/3x), \label{e:ome2-bounds} \\
  x\log_3(1/x) \le\, & \ome_3(x) \le x\log_3(3/2x), \label{e:ome3-bounds} \\
  x\log_4(1/x) \le\, & \ome_4(x) \le x\log_4(4/3x), \label{e:ome4-bounds} \\
\intertext{and for $m\ge 5$,}
  x\log_m(e/mx) \le\, & \ome_m(x) \le x\log_m(3/2x), \label{e:omem-bounds}
\end{align}
for any $x\in(0,1]$. We notice that estimates
\refe{ome2-bounds}--\refe{ome4-bounds} are sharp: the lower bound in
\refe{ome2-bounds} and \refe{ome4-bounds} is attained for $x=2^{-k}$ and the
upper bound for $x=2^{1-k}/3$, the lower bound in \refe{ome3-bounds} is
attained for $x=3^{-1-k}$ and the upper bound for $x=3^{-k}/2$, for any
integer $k\ge 0$. In contrast, the estimate \refe{omem-bounds} is not sharp;
it is provided, essentially, as a ``proof of concept'' and can easily be
improved. However, as $x\to 0$, the lower and upper bounds in
\refe{omem-bounds} coincide up to lower-order terms, and it may well be
impossible to obtain both sharp and explicit bounds of this sort for
 $m\ge 5$.

The graphs of the functions $\ome_2,\ome_3$, and $\ome_5$, along with the
functions representing the corresponding lower and upper bounds, are shown in
Figure 1.

\begin{figure}[h]
\centering
\includegraphics[trim=1.5mm 1.5mm 1.5mm 1.5mm,clip,
  height=0.31\textwidth,width=0.25\textwidth,angle=-90]{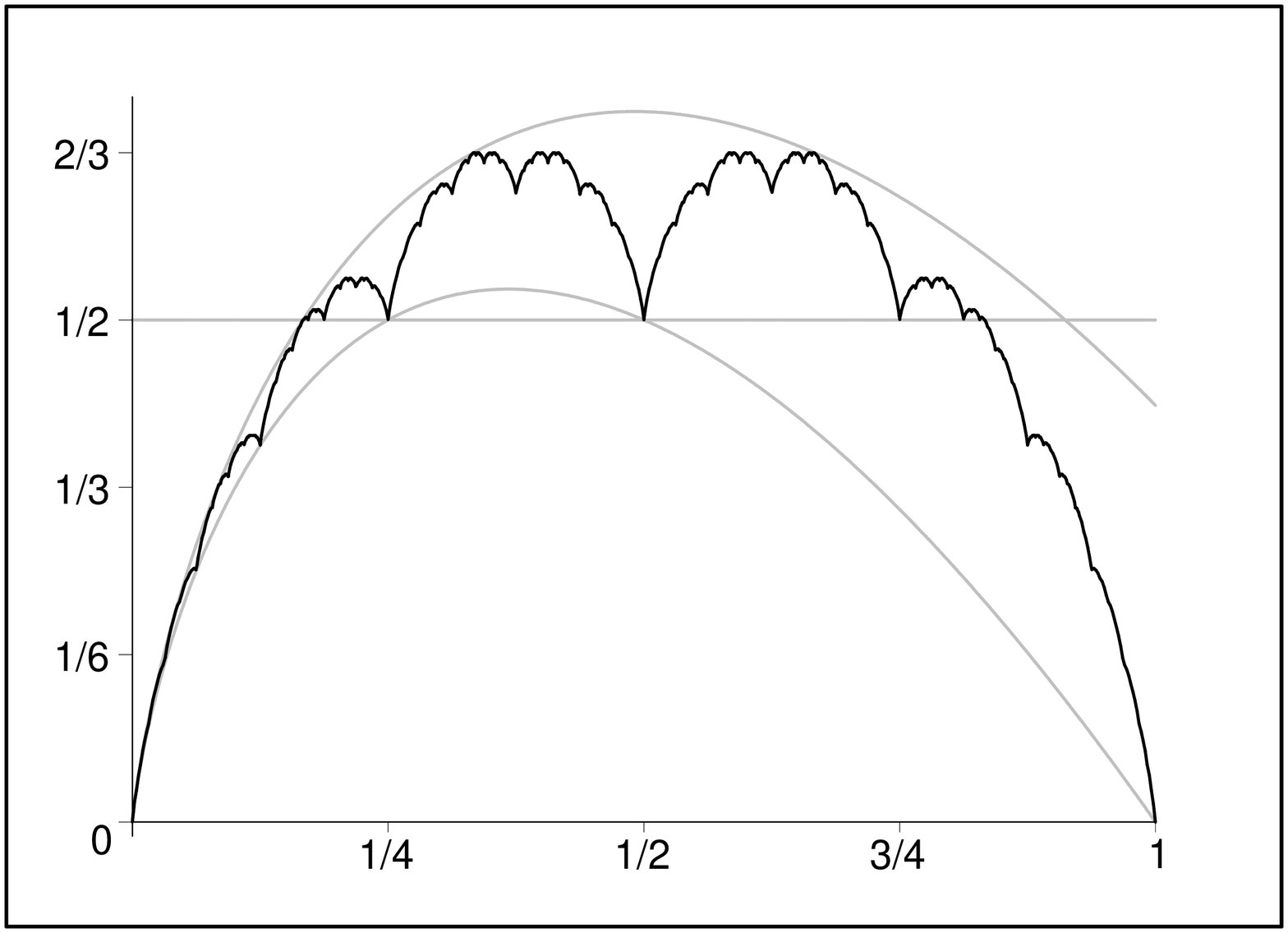}
\hskip .25cm
\includegraphics[trim=1.5mm 1.5mm 1.5mm 1.5mm,clip,
  height=0.31\textwidth,width=0.25\textwidth,angle=-90]{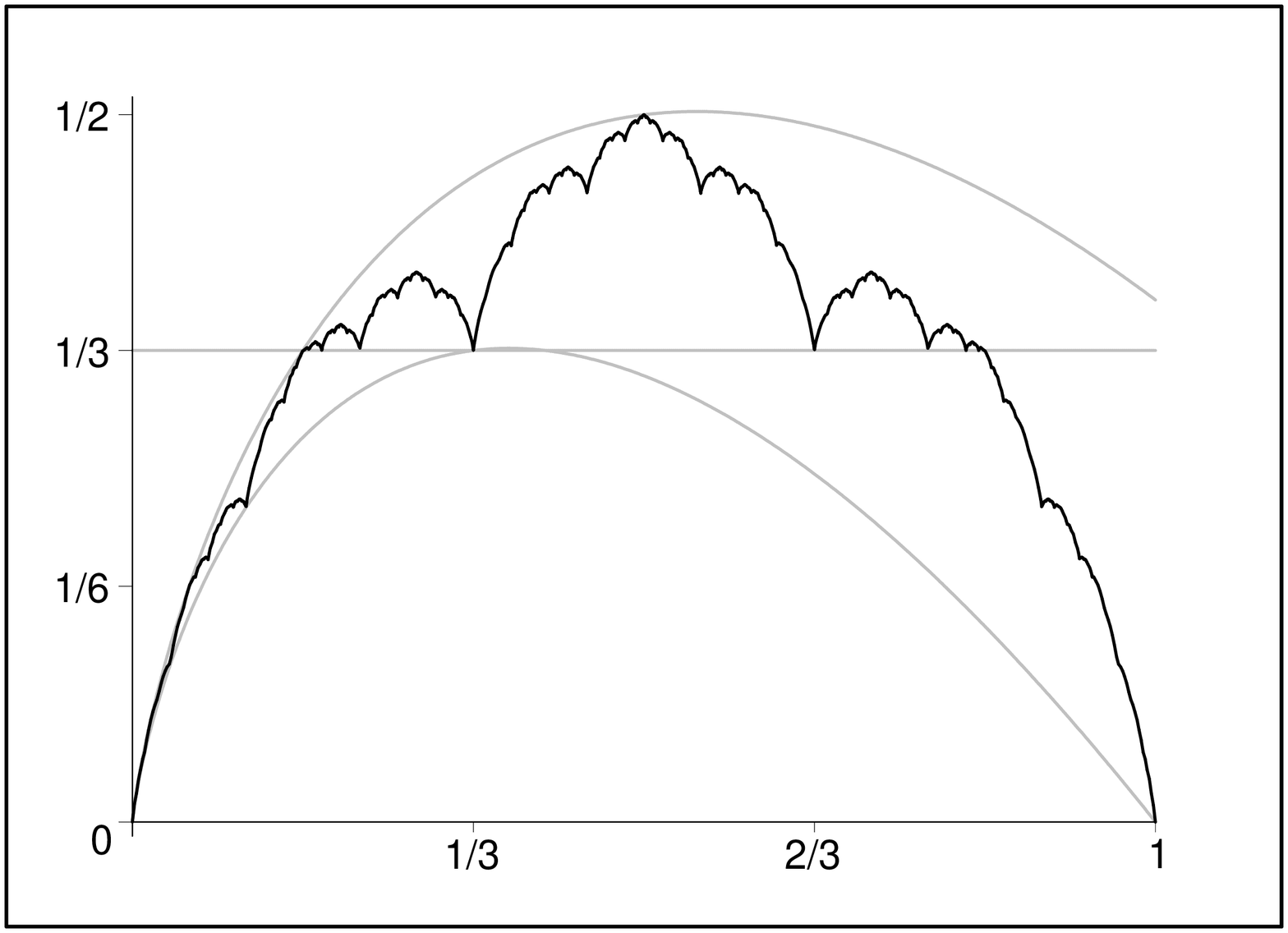}
\hskip .25cm
\includegraphics[trim=1.5mm 1.5mm 1.5mm 1.5mm,clip,
  height=0.31\textwidth,width=0.25\textwidth,angle=-90]{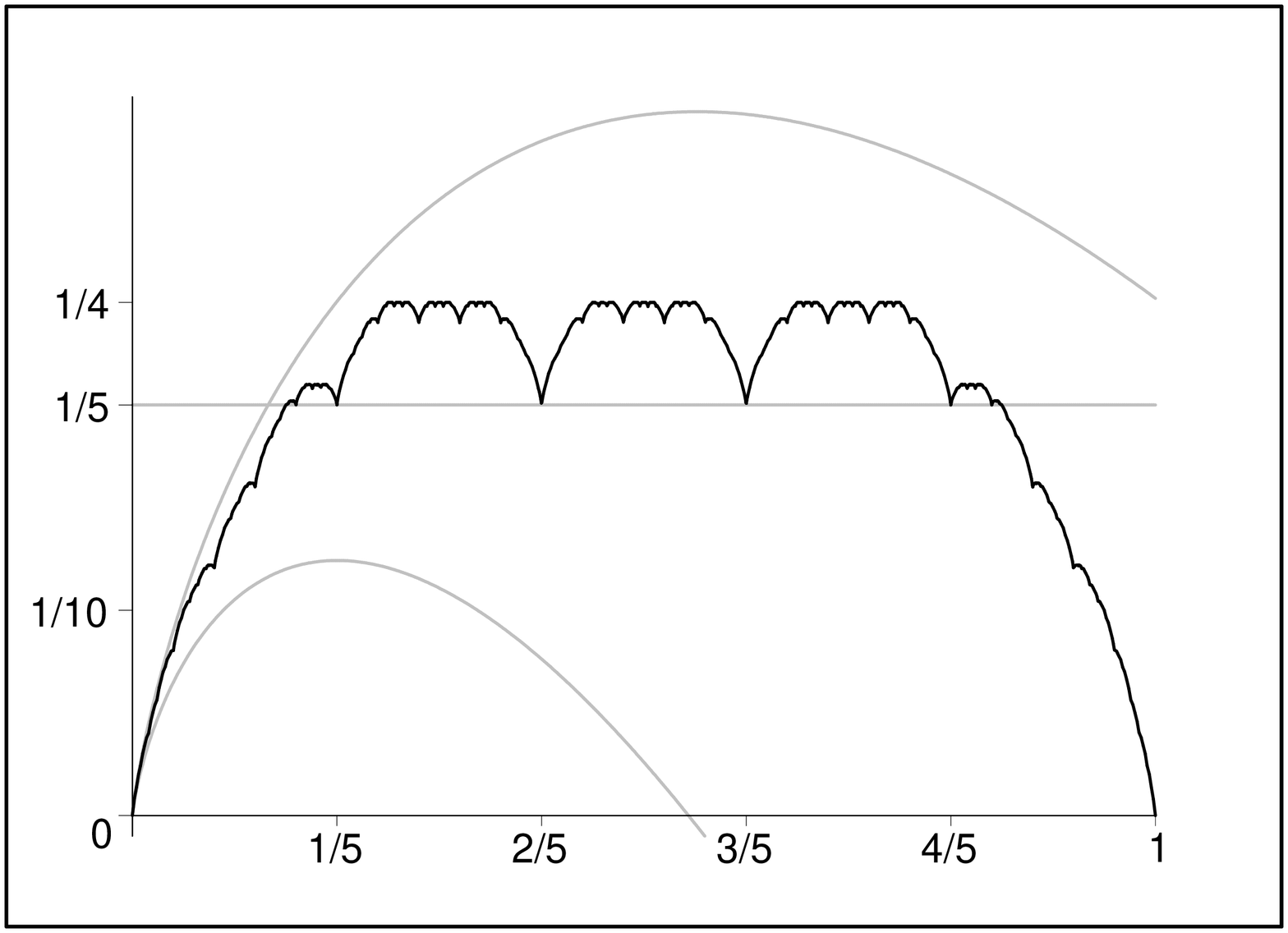}
\caption{The graphs of the functions $\ome_2$ (Takagi function), $\ome_3$,
  and $\ome_5$.}
\end{figure}


The reason for omitting the graph of $\ome_4$ is that this function turns out
to be identical, up to a constant factor, to the Takagi function; namely, we
have
\begin{equation}\label{e:omega4=omega2}
  \ome_4 = \frac12\,\ome_2.
\end{equation}
To prove this somewhat surprising relation, it suffices to show that for any
real $x$ and integer $k\ge 0$ we have
  $$ 2^{-2k}\|2^{2k}x\|+2^{-2k-1}\|2^{2k+1}x\|
             = 2\cdot 4^{-k} \|4^kx\|_{1/4}. $$
Indeed, letting $z:=2^{2k}x$ and multiplying by $2^{2k+1}$, we can rewrite
this equality as
  $$ 2\|z\|+\|2z\|=4\|z\|_{1/4}, $$
and this is readily verified by restricting $z$ to the range $0\le z\le 1/2$
and considering separately the cases $z\le 1/4$ and $z\ge 1/4$.

It was conjectured by P\'ales \refb{p} and proved by Boros \refb{bo} that the
Takagi function satisfies
\begin{equation}\label{e:bp}
  \ome_2\bfr{x_1+x_2}{}
              \le \frac{\ome_2(x_1)+\ome_2(x_2)}2 + \frac12\,|x_1-x_2|,
                                                          \quad x_1,x_2\in\R.
\end{equation}
Combining this inequality with a result of H\'azy and P\'ales
\cite[Theorem~4]{b:hp}, one immediately derives the following stronger
version:
\begin{multline}\label{e:gen-bp}
  \ome_2(\lam x_1+(1-\lam)x_2)
     \le \lam\ome_2(x_1)+(1-\lam)\ome_2(x_2) + \ome_2(\lam) |x_1-x_2|, \\
                                                 x_1,x_2\in\R,\ \lam\in[0,1].
\end{multline}
We give short proofs to \refe{bp} and \refe{gen-bp}, which seem to be
genuinely different from the original proofs, and establish the $3$-adic
analogs.

\begin{theorem}\label{t:bp3}
We have
  $$ \ome_3\(\frac{x_1+x_2+x_3}3\)
       \le \frac{\ome_3(x_1)+\ome_3(x_2)+\ome_3(x_3)}3 + \frac13\,(x_3-x_1) $$
for any real $x_1\le x_2\le x_3$.
\end{theorem}

\begin{theorem}\label{t:gen-bp3}
We have
  $$ \ome_3(\lam x_1+(1-\lam)x_2)
       \le \lam\ome_3(x_1)+(1-\lam)\ome_3(x_2) + \ome_3(\lam)\,|x_2-x_1| $$
for any real $x_1,x_2$, and $\lam\in[0,1]$.
\end{theorem}

The importance of Boros-P\'ales inequality~\refe{bp} and Theorem~\reft{bp3}
for our present purposes, and the way they are applied in this paper, will be
explained shortly. Inequality \refe{gen-bp} and Theorem~\reft{gen-bp3} are
derived as particular cases of a more general result, presented below
(Theorem~\reft{funny}).

Back to Theorem~\reft{main1}, we actually prove a more versatile and precise
result, with the improvement being particularly significant for small values
of $m$. To state it we bring into consideration the classes of functions,
defined as follows. For integer $m\ge 2$, let $\cF_m$ consist of all
real-valued functions $f$, defined on the interval $[0,1]$, satisfying the
boundary condition
\begin{equation}\label{e:bound-cond}
  \max\{f(0),f(1)\}\le 0,
\end{equation}
and such that for any $x_1\longc x_m\in[0,1]$ with $\min_i x_i=x_1$ and
 $\max_i x_i=x_m$, we have
\begin{equation}\label{e:m}
  f\( \frac{x_1\longp x_m}m \) \le \frac{f(x_1)\longp f(x_m)}{m} + (x_m-x_1).
\end{equation}
Condition \refe{m} can be understood as a ``relaxed convexity'' and in fact,
any convex function satisfying the boundary condition \refe{bound-cond} is
contained in every class $\cF_m$.

We notice that if $l,m\ge 2$ are integers with $l\mid m$, then
$\cF_m\seq\cF_l$: for, given a function $f\in\cF_m$ and a system of $l$
numbers in $[0,1]$, we can ``blow up'' this system to get a system of $m$
numbers (where every original number is repeated $m/l$ times) and apply then
\refe{m} to this new system to obtain the analogue of \refe{m} for the
original $l$ numbers. For $l=2$ and $m=4$ the inverse inclusion holds, too,
so that we have $\cF_4=\cF_2$; to prove this, fix $f\in\cF_2$ and
$x_1,x_2,x_3,x_4\in[0,1]$ with $x_1\le x_2\le x_3\le x_4$, and observe that
then
\begin{align*}
  f\(\frac{x_1+x_2+x_3+x_4}4\)
     &\le \frac12\, \( f\(\frac{x_1+x_2}2\) +  f\(\frac{x_3+x_4}2\) \) \\
     &\hskip 2.0in         + \frac{x_3+x_4}2 - \frac{x_1+x_2}2 \\
     &\le \frac14\,\big( f(x_1)+f(x_2)+f(x_3)+f(x_4) \big) + (x_4-x_1),
\end{align*}
whence $f\in\cF_4$.

It is not difficult to see, however, that, say, the class $\cF_6$ is distinct
from each of the classes $\cF_2$ and $\cF_3$, and the class $\cF_8$ is
distinct from the class $\cF_2$. Indeed, a straightforward numerical
verification confirms that the functions $F_2\in\cF_2$ and $F_3\in\cF_3$,
introduced below in this section, satisfy $F_2\notin\cF_6,\ F_3\notin\cF_6$,
and $F_2\notin\cF_8$.

We notice that Boros-P\'ales inequality \refe{bp} can be interpreted as
$2\ome_2\in\cF_2$, and Theorem~\reft{bp3} gives $3\ome_3\in\cF_3$. (In fact,
it is the \emph{restrictions} of the functions $2\ome_2$ and $3\ome_3$ onto
the interval $[0,1]$ that belong to the classes $\cF_2$ and $\cF_3$,
respectively. However, this little abuse of notation does not lead to any
confusion.)

\begin{theorem}\label{t:isoper}
Let $m\ge 2$ be an integer. Suppose that $f\in\cF_m$, and $G$ is a finite
abelian group, the exponent of which divides $m$. Then for any subset $A\seq
G$ and any generating subset $S\seq G$ we have
  $$ \prt_S(A) \ge \frac 1m\,|G|\, f(|A|/|G|). $$
\end{theorem}

Theorem~\reft{main1} is an immediate consequence of Theorem~\reft{isoper} and
\begin{proposition}\label{p:exln1x}
Let $f(x)=ex\ln(1/x)$ for $x\in(0,1]$, and $f(0)=0$. Then $f\in\cF_m$ for any
integer $m\ge 2$.
\end{proposition}

To use Theorem~\reft{isoper} more efficiently, we must choose the function
$f$ in an optimal way for every particular value of $m$. Our next result
shows that for each $m\ge 2$, there is a ``universal'' choice which does not
depend on the density $|A|/|G|$.

\begin{theorem}\label{t:Fmp}
For any $m\ge 2$ there is a (unique) function $F_m\in\cF_m$ such that for any
other function $f\in\cF_m$ and any $x\in[0,1]$ we have $F_m(x)\ge f(x)$. The
function $F_m$ is continuous on $[0,1]$, strictly positive on $(0,1)$, and
satisfies $F_m(0)=F_m(1)=0$ and $F_m(x)=F_m(1-x)$ for $x\in[0,1]$.
\end{theorem}

From now on we adopt $F_m$ as a standard notation for the functions of
Theorem~\reft{Fmp}.

We were able to find the functions $F_m$ explicitly for $m\in\{2,3,4\}$,
and estimate them for $m\ge 5$. Determining $F_m$ for $m\ge 5$ seems to
be a non-trivial and challenging problem; we have done some work towards
the case $m=5$, and the results may appear elsewhere.
\begin{theorem}\label{t:Fm}
For any $m\ge 2$ we have $F_m\le m\ome_m$, with equality if and only if
$m\in\{2,3,4\}$.
\end{theorem}

The case $m\in\{2,3,4\}$ of Theorem~\reft{Fm} will be derived from \refe{bp},
Theorem~\reft{bp3}, and \refe{omega4=omega2}.

As remarked above, \refe{min-prt} follows from the results of the present
paper; indeed, the reader can now see that it is an immediate corollary of
Theorems~\reft{lex}, \reft{isoper}, and \reft{Fm}.

Combining Theorems \reft{isoper} and \reft{Fm} and estimates
\refe{ome2-bounds}--\refe{ome4-bounds} we obtain the following result.
\begin{corollary}
If $G$ is a finite abelian group of exponent $m\in\{2,3,4\}$, then for any
non-empty subset $A\seq G$ and generating subset $S\seq G$ we have
  $$ \prt_S(A) \ge|G|\,\ome_m(|A|/|G|) \ge |A|\log_m \frac{|G|}{|A|}. $$
\end{corollary}
We remark that in the case $m=2$, the resulting estimate
$\prt_S(A)\ge|A|\log_2(|G|/|A|)$ is well-known, the first appearance in the
literature we are aware of being \cite[Lemma~4.1]{b:cfgs}.

Theorem~\reft{Fm} can be considerably improved for large values of $m$.
\begin{proposition}\label{p:Fm-upper}
For any integer $m\ge 2$ and real $x\in(0,1]$ we have
  $$ F_m(x) \le \frac{m}{m-1}\,ex\ln(e/x). $$
\end{proposition}

For the lower bound, we notice that Proposition~\refp{exln1x} yields
  $$ F_m(x)\ge ex\ln(1/x), $$
for each $m\ge 2$ and $x\in(0,1]$.

We deduce Proposition~\refp{Fm-upper} from the following result which, in
view of Theorem~\reft{Fm}, generalizes \refe{gen-bp} and
Theorem~\reft{gen-bp3}.
\begin{theorem}\label{t:funny}
Let $m\ge 2$ be an integer. Then for any function $f\in\cF_m$ and any
$\lam,x,y\in[0,1]$ with $x\le y$ we have
\begin{equation}\label{e:gen-bp-m}
  f(\lam x+(1-\lam)y) \le \lam f(x)+(1-\lam)f(y) + (y-x)F_m(\lam).
\end{equation}
Moreover, if $f$ is a function, defined on the whole real line and satisfying
\refe{m} for any real $x_1\longc x_m$ with $\min_i x_i=x_1$ and $\max_i
x_i=x_m$, then \refe{gen-bp-m} holds for any real $x\le y$ and
$\lam\in[0,1]$.
\end{theorem}

The rest of the paper, devoted to the proof of the results discussed above,
is partitioned into three sections and appendix. In Section \refs{Takagi} we
study the generalized Takagi functions, proving Boros-P\'ales inequality
\refe{bp} and its $3$-adic analog, Theorem~\reft{bp3}, and establishing an
important lemma used in both proofs and also in the proofs of
Theorems~\reft{lex} and~\reft{Fm}. Section ~\refs{isoper} deals with the
isoperimetric problem: we prove here Theorems~\reft{lex} and \reft{isoper}.
In Section \refs{Fm} we investigate the classes $\cF_m$ and the functions
$F_m$, and prove Propositions~\refp{exln1x} and~\refp{Fm-upper}, and
Theorems~\reft{Fmp}, \reft{Fm}, and~\reft{funny}. As remarked above,
Theorem~\reft{main1} is an immediate consequence of Theorem~\reft{isoper} and
Proposition~\refp{exln1x}, while \refe{gen-bp} and Theorem~\reft{gen-bp3} (in
view of Theorem~\reft{Fm}) are particular cases of Theorem~\reft{funny};
hence no additional proofs are needed. In the Appendix we prove estimates
\refe{ome2-bounds}--\refe{omem-bounds}.

\section{The generalized Takagi functions: proofs of \\
  the Boros-P\'ales inequality and Theorem~\reft{bp3}}\label{s:Takagi}

The following lemma, used in the proofs of the Boros-P\'ales inequality and
Theorems~\reft{lex}, \reft{bp3}, and~\reft{Fm}, is known in the case $m=2$;
see \refb{hy} or \cite[Theorem~5.1]{b:ak}.
\begin{lemma}\label{l:takagi}
Let $m\ge 2$ be an integer. Then for any integer $r\ge 1$ and $n$, the latter
of which is not divisible by $m$, writing $n=tm+\rho$ with integer $t$ and
$\rho\in[1,m-1]$, we have
  $$ \ome_m\(\frac n{m^r}\) = \(1-\frac{\rho}m\)\ome_m\(\frac t{m^{r-1}}\)
                       + \frac{\rho}m\,\ome_m\(\frac{t+1}{m^{r-1}}\)
                                                      \, + \, \frac1{m^r}. $$
\end{lemma}

\begin{proof}
We want to prove that
  $$ \sum_{k=0}^\infty m^{-k} \( \|m^{k-r}n\|_{1/m}
        - \(1-\frac{\rho}m\)\|m^{k+1-r}t\|_{1/m}
               - \frac{\rho}m \,\|m^{k+1-r}(t+1)\|_{1/m} \) = \frac1{m^r}. $$
We notice that all the summands in the left-hand side, corresponding to
 $k\ge r$, vanish, while the summand, corresponding to $k=r-1$, contributes
$m^{-(r-1)}(1/m)=m^{-r}$ to the sum. Consequently, to complete the proof it
suffices to show that
  $$ m\|m^{k-r}n\|_{1/m}
          = (m-\rho)\|m^{k+1-r}t\|_{1/m}+\rho\,\|m^{k+1-r}(t+1)\|_{1/m},
                                                       \quad k\in[0,r-2]. $$
To this end we prove that the interval $(m^{k+1-r}t,m^{k+1-r}(t+1))$ (of
which $m^{k-r}n$ is an internal point) does not contain any number of the
form $N+\eps/m$ with integer $N$ and $\eps\in\{0,\pm1\}$, and therefore
$\|x\|_{1/m}$ is a linear function of $x$ on this interval. Indeed, if we had
  $$ m^{k+1-r}t < N + \eps/m < m^{k+1-r}(t+1), $$
then, multiplying by $m^{r-k-1}$, we would get $t<Nm^{r-k-1}+\eps
m^{r-k-2}<t+1$, which cannot hold since the midterm is an integer.
\end{proof}

For the rest of this section, for integer $n$ and $m\ge 2$ we let
  $$ \del_m(n) := \begin{cases}
                     0 &\ \text{if $m$ divides $n$}, \\
                     1 &\ \text{if $m$ does not divide $n$}.
                   \end{cases} $$
We record the following immediate corollary of Lemma \refl{takagi}.
\begin{corollary}\label{c:takagi}
For any integer $r\ge 1$ and $n$, we have
  $$ \ome_2\(\frac{n}{2^r}\)
       = \frac12\, \ome_2\(\frac{\lfl n/2\rfl}{2^{r-1}}\)
             + \frac12\, \ome_2\(\frac{\lcl n/2\rcl}{2^{r-1}}\)
                                        \, + \, \frac{\del_2(n)}{2^r}, $$
where $\lfl\cdot\rfl$ and $\lcl\cdot\rcl$ denote the floor and the ceiling
functions, respectively.
\end{corollary}

\begin{proof}[Proof of Boros-P\'ales inequality \refe{bp}]
Since $\ome_2$ is a continuous function, it suffices to show that
  $$ \ome_2\bfr{x+y}{r}
          \le \frac12\, \ome_2\bfr{x}{r-1} + \frac12\, \ome_2\bfr{y}{r-1}
                                                     + \frac1{2^r}\,|x-y|, $$
for any integer $x,y$, and $r\ge 1$. We use induction on $r$. The case $r=1$
is immediate since $\ome_2(x)=\ome_2(y)=0$ and $\ome_2((x+y)/2)$ is equal to
$0$ or $1$ depending on whether $x$ and $y$ are of the same or of distinct
parity. Thus, we assume that $r\ge 2$. Moreover, we assume that $x$ is odd;
clearly, this does not restrict the generality.

Applying Corollary \refc{takagi} with $n=x+y$ and using the induction
hypothesis, we get
\begin{align*}
  \ome_2\bfr{x+y}{r}
    &= \frac12\, \ome_2\bfr{\lfl(x+y)/2\rfl}{r-1}
                 + \frac12\, \ome_2\bfr{\lcl(x+y)/2\rcl}{r-1}
                                              + \frac1{2^r}\,\del_2(x+y) \\
    &=  \frac12\, \ome_2\bfr{(x-1)/2+\lfl(y+1)/2\rfl}{r-1}
                 + \frac12\, \ome_2\bfr{(x+1)/2+\lcl(y-1)/2\rcl}{r-1} \\
    &{\hskip 4.0in} + \frac1{2^r}\,\del_2(x+y) \\
    &\le \frac14\, \( \ome_2\bfr{x-1}{r-1} + \ome_2\bfr{x+1}{r-1} \) \\
    &{\hskip 0.6in}
           + \frac14\, \( \ome_2\bfr{2\lfl(y+1)/2\rfl}{r-1}
              + \ome_2\bfr{2\lcl(y-1)/2\rcl}{r-1} \) \\
    &{\hskip 0.6in}
      + \frac1{2^r}\,\left| \lfl \frac{y+1}2\rfl-\frac{x-1}2 \right|
      + \frac1{2^r}\,\left| \lcl \frac{y-1}2\rcl-\frac{x+1}2 \right|
      + \frac1{2^r}\,\del_2(x+y).
\end{align*}
We now notice that, by Corollary \refc{takagi},
  $$ \frac12\, \( \ome_2\bfr{x-1}{r-1} + \ome_2\bfr{x+1}{r-1} \)
                                 = \ome_2\bfr{x}{r-1}\,-\,\frac1{2^{r-1}} $$
and
  $$ \frac12\, \( \ome_2\bfr{2\lfl(y+1)/2\rfl}{r-1}
           + \ome_2\bfr{2\lcl(y-1)/2\rcl}{r-1} \)
                = \ome_2\bfr{y}{r-1} - \frac{\del_2(y)}{2^{r-1}}, $$
as it follows easily by considering separately the cases of even and odd $y$.
Consequently,
\begin{align*}
  \ome_2\bfr{x+y}{r}
    &\le \frac12\,\ome_2\bfr{x}{r-1}+\frac12\,\ome_2\bfr{y}{r-1}
                     + \frac1{2^r}\,(\del_2(x+y)-\del_2(y)-1) \\
    &{\hskip 2in} + \frac1{2^r}\, \left| \lfl \frac{y-x+2}2 \rfl \right|
                  + \frac1{2^r}\, \left| \lcl \frac{y-x-2}2 \rcl \right|.
\end{align*}
To complete the proof we observe that $\del_2(x+y)-\del_2(y)-1\le 0$, and
that if $x\neq y$ (in which case the assertion is trivial), then
$\lfl(y-x+2)/2\rfl$ and $\lcl(y-x-2)/2\rcl$ are of the same sign, whence
\begin{align*}
  |\lfl(y-x+2)/2\rfl| + |\lcl(y-x-2)/2\rcl|
     &= |\lfl(y-x+2)/2\rfl + \lcl(y-x-2)/2\rcl| \\
     &= |\lfl(y-x)/2\rfl+\lcl(y-x)/2\rcl| \\
     &= |y-x|.
\end{align*}
\end{proof}

To prove Theorem~\reft{bp3} we need yet another corollary of
Lemma~\refl{takagi}.
\begin{corollary}\label{c:takagi3} Let $r\ge 1$ and $n$ be
integers. If $\xi_n\in\{-1,0,1\}$ and $\zet_n\in\{-2,0,2\}$ are defined by
$n\equiv\xi_n\equiv\zet_n\mmod 3$, then
  $$ \ome_3\tfr{n}{r} = \frac23\, \ome_3\tfr{n-\xi_n}{r}
           + \frac13\, \ome_3\tfr{n-\zet_n}{r} + \frac{\del_3(n)}{3^r}. $$
\end{corollary}
Observe that, with $\xi_n$ and $\zet_n$ defined as in Corollary
\refc{takagi3}, we have
\begin{equation}\label{e:xizet}
  2\xi_n+\zet_n=0.
\end{equation}

\begin{proof}[Proof of Theorem~\reft{bp3}]
By continuity of $\ome_3$, it suffices to show that
  $$ \ome_3\tfr{x+y+z}{r}
       \le \frac13\, \ome_3\tfr{x}{r-1} + \frac13\, \ome_3\tfr{y}{r-1}
               + \frac13\, \ome_3\tfr{z}{r-1} + \frac1{3^r}\,(z-x), $$
for any integer $r\ge 1$ and $x\le y\le z$.

For integer $r\ge 0$ and $n$, let
  $$ T_r(n) := \sum_{k=1}^r 3^k \, \|3^{-k}n\|_{1/3}. $$
Thus, $T_0(n)=0$, $T_1(n)=\del_3(n)$, $T_r(-n)=T_r(n)$, and
$T_r(3n)=3T_{r-1}(n)$; these simple observations may be used below without
special references. Furthermore,
  $$ 3^r \ome_3\tfr{n}{r}
               = \sum_{k=0}^{r-1} 3^{r-k} \|3^{k-r} n\|_{1/3} = T_r(n); $$
therefore, keeping the notation of Corollary \refc{takagi3}, we can re-write
its conclusion as
\begin{equation}\label{e:Txizet}
  T_r(n) = \frac23\, T_r(n-\xi_n) + \frac13\, T_r(n-\zet_n) + \del_3(n),
\end{equation}
and the estimate we have to prove as
\begin{equation}\label{e:toprove}
  T_r(x+y+z)
                  \le T_{r-1}(x)+T_{r-1}(y)+T_{r-1}(z)+(z-x).
\end{equation}
To establish \refe{toprove} we use induction on $r$. For $r=1$ the assertion
is easy to verify in view of $T_0=0$ and $T_1(x+y+z)=\del_3(x+y+z)$, and we
assume that $r\ge 2$. We also assume that $x$ is strictly smaller than $z$;
for if $x=y=z$, then \refe{toprove} is immediate from $T_r(3x)=3T_{r-1}(x)$.

If $x,y$, and $z$ are all divisible by $3$, then the assertion follows easily
from the induction hypothesis. Otherwise, changing (simultaneously) the signs
of $x,y$, and $z$, if necessary, we can assume that one of the following
holds:
\begin{itemize}
\item[(i)]   $x\equiv y\equiv z\equiv 1\mmod 3$;
\item[(ii)]  two of the numbers $x,y$, and $z$ are congruent to $1$
    modulo $3$, and the third is divisible by $3$;
\item[(iii)] the numbers $x,y$, and $z$ are pairwise incongruent
    modulo $3$;
\item[(iv)]  two of the numbers $x,y$, and $z$ are divisible by $3$, and
    the third is congruent to $1$ modulo $3$;
\item[(v)]   two of the numbers $x,y$, and $z$ are congruent to $1$
    modulo $3$, and the third is congruent to $2$ modulo $3$.
\end{itemize}
We consider these five cases separately.

\subsection*{Case (i): $x\equiv y\equiv z\equiv 1\mmod 3$.}
In this case, using the induction hypothesis we get
\begin{align}\label{e:casei}
  T_r(x+y+z)
    &=   3T_{r-1}\(\frac{x-1}3+\frac{y-1}3+\frac{z+2}3\) \notag \\
    &\le 3T_{r-2}\(\frac{x-1}3\)+3T_{r-2}\(\frac{y-1}3\)
                                   + 3T_{r-2}\(\frac{z+2}3\) \notag \\
    &{}\hskip 3.25in  +(z-x+3) \notag \\
    &=   T_{r-1}(x-1)+T_{r-1}(y-1)+T_{r-1}(z+2) + (z-x+3).
\end{align}
Similarly,
\begin{equation}\label{e:caseia}
  T_r(x+y+z)
           \le T_{r-1}(x-1)+T_{r-1}(y+2)+T_{r-1}(z-1) + (z-x)
\end{equation}
and
\begin{equation}\label{e:caseib}
  T_r(x+y+z)
           \le T_{r-1}(x+2)+T_{r-1}(y-1)+T_{r-1}(z-1) + (z-x-3),
\end{equation}
except that we must add $3$ to the right-hand side of \refe{caseia} if
$y=z$, and to the right-hand side of \refe{caseib} if $x=y$.
Averaging \refe{casei}--\refe{caseib} and taking into account the observation
just made and the fact that if $n\equiv 1\pmod 3$, then
  $$ \frac23\,T_{r-1}(n-1)+\frac13\,T_{r-1}(n+2) = T_{r-1}(n) - 1 $$
(as it follows from \refe{Txizet}), we get \refe{toprove}.

\subsection*{Case (ii): two of $x,y$, and $z$ are congruent to $1$
  modulo $3$, and the third is divisible by $3$.}

Denote by $w$ the element of the set $\{x,y,z\}$ which is divisible by $3$,
and let $u$ be the smallest, and $v$ the largest of the two other elements.
By \refe{Txizet}, we have
\begin{align}\label{e:caseii}
  T_r(x+y+z) &= \frac23\, T_r\(x+y+z+1\)
                          + \frac13\,T_r\(x+y+z-2\) + 1 \notag \\
             &= 2 T_{r-1}\(\frac{u+v+w+1}3\)
                          + T_{r-1}\(\frac{u+v+w-2}3\) + 1.
\end{align}
By the induction hypothesis,
\begin{align}\label{e:caseiia}
  T_{r-1}\(\frac{u+v+w+1}3\)
    &=   T_{r-1}\(\frac{u-1}3 + \frac{v+2}3
                                            + \frac{w}3\) \notag \\
    &\le T_{r-2}\(\frac{u-1}3\)+T_{r-2}\(\frac{v+2}3\)
                                   + T_{r-2}\(\frac{w}3\) \notag \\
    &{}\hskip 2.5in                + \frac{z-x+3}3 \notag \\
    &=   \frac13\,T_{r-1}(u-1) + \frac13\,T_{r-1}(v+2)
                                   + \frac13\, T_{r-1}(w) \notag \\
    &{}\hskip 2.5in               + \frac{z-x+3}3
\end{align}
and similarly,
\begin{multline}\label{e:caseiib}
  T_{r-1}\(\frac{u+v+w+1}3\)
    \le \frac13\,T_{r-1}(u+2) + \frac13\,T_{r-1}(v-1)
                                   + \frac13\, T_{r-1}(w) \\
                                                  + \frac{z-x+3}3.
\end{multline}
Also,
\begin{align}\label{e:caseiic}
  T_{r-1}\(\frac{u+v+w-2}3\)
    &=   T_{r-1}\(\frac{u-1}3 + \frac{v-1}3
                                            + \frac{w}3\) \notag \\
    &\le T_{r-2}\(\frac{u-1}3\)+T_{r-2}\(\frac{v-1}3\)
                                   + T_{r-2}\(\frac{w}3\) \notag \\
    &{}\hskip 2.5in                + \frac{z-x+1}3 \notag \\
    &=   \frac13\,T_{r-1}(u-1) + \frac13\,T_{r-1}(v-1)
                                   + \frac13\, T_{r-1}(w) \notag \\
    &{}\hskip 2.5in               + \frac{z-x+1}3.
\end{align}

In fact, we need a slight refinement of \refe{caseiia}--\refe{caseiic} which
can be obtained by distinguishing the subcases where $w=x$ (meaning that it
is the smallest of the numbers $x,y,z$ that is divisible by $3$), $w=y$ (the
middle one is divisible by $3$) and $w=z$ (the largest one is divisible by
$3$). The reader will check it easily that in the first case ($w=x$), both
last summands in the right-hand sides of \refe{caseiia} and \refe{caseiib}
can be replaced with $(z-x+2)/3$, and the last summand in the right-hand side
of \refe{caseiic} can be replaced with $(z-x-1)/3$. Similarly, in the second
case ($w=y$), we can replace the last summands in the right-hand sides of
both \refe{caseiib} and \refe{caseiic} with $(z-x)/3$, and in the third case
($w=z$), both last summands in the right-hand sides of \refe{caseiia} and
\refe{caseiib} can be replaced with $(z-x+1)/3$. In any case, the sum of the
three summands does not exceed $z-x+1$. Taking this into account, adding up
\refe{caseiia}--\refe{caseiic}, and substituting the result into
\refe{caseii}, we get
\begin{multline*}
  T_r(x+y+z) \le \( \frac23\,T_{r-1}(u-1) + \frac13\,T_{r-1}(u+2) \) \\
             +  \( \frac23\,T_{r-1}(v-1) + \frac13\,T_{r-1}(v+2) \)
                 + T_{r-1}(w) + (z-x) + 2.
\end{multline*}
The result now follows from \refe{Txizet}.

\subsection*{Case (iii): $x,y$, and $z$ are pairwise incongruent modulo
              $3$.}

Using the induction hypothesis and the fact that
$\xi_x+\xi_y+\xi_z=\zet_x+\zet_y+\zet_z=0$ we obtain in this case
\begin{align}\label{e:caseiiia}
  T_r(x+y+z)
    &=   3T_{r-1}\(\frac{x-\xi_x}3+\frac{y-\xi_y}3+\frac{z-\xi_z}3\)
                                                                   \notag \\
    &\le 3T_{r-2}\(\frac{x-\xi_x}3\)+3T_{r-2}\(\frac{y-\xi_y}3\)
                                   + 3T_{r-2}\(\frac{z-\xi_z}3\) \notag \\
    &{}\hskip 2.75in        +(z-x-\xi_z+\xi_x) \notag \\
    &=   T_{r-1}(x-\xi_x)+T_{r-1}(y-\xi_y)+T_{r-1}(z-\xi_z) \notag \\
    &{}\hskip 2.75in        + (z-x-\xi_z+\xi_x).
\end{align}
Similarly,
\begin{multline}\label{e:caseiiib}
  T_r(x+y+z) \le T_{r-1}(x-\zet_x)+T_{r-1}(y-\zet_y)
                                                      +T_{r-1}(z-\zet_z) \\
       + (z-x-\zet_z+\zet_x+6),
\end{multline}
for $\max\{x-\zet_x,y-\zet_y,z-\zet_z\}\le z-\zet_z+3$ and
$\min\{x-\zet_x,y-\zet_y,z-\zet_z\}\ge x-\zet_x-3$. The assertion follows by
averaging \refe{caseiiia} and \refe{caseiiib} with the weights $2/3$ and
$1/3$, respectively, using \refe{Txizet}, and noticing that
\begin{multline*}
  -\del_3(x)-\del_3(y)-\del_3(z)
         + \frac23\,(-\xi_z+\xi_x)+\frac13\,(-\zet_z+\zet_x+6) \\
               = \frac13\,(\zet_x+2\xi_x) - \frac13\,(\zet_z+2\xi_z) = 0.
\end{multline*}

\subsection*{Case (iv): two of $x,y$, and $z$ are divisible by $3$, and the
  third is congruent to $1$ modulo $3$.}

By \refe{Txizet}, we have
\begin{equation}\label{e:caseiv}
  T_r(x+y+z) = 2 T_{r-1}\(\frac{x+y+z-1}3\)
                                 + T_{r-1}\(\frac{x+y+z+2}3\) + 1.
\end{equation}
By the induction hypothesis,
\begin{align}\label{e:caseiva}
  T_{r-1}\(\frac{x+y+z-1}3\)
    &=   T_{r-1}\(\frac{x-\xi_x}3 + \frac{y-\xi_y}3
                                            + \frac{z-\xi_z}3\) \notag \\
    &\le T_{r-2}\(\frac{x-\xi_x}3\)+T_{r-2}\(\frac{y-\xi_y}3\)
                                   + T_{r-2}\(\frac{z-\xi_z}3\) \notag \\
    &{}\hskip 2.2in                + \frac{z-x-\xi_z+\xi_x}3 \notag \\
    &=   \frac13\,T_{r-1}(x-\xi_x) + \frac13\,T_{r-1}(y-\xi_y)
                                   + \frac13\, T_{r-1}(z-\xi_z) \notag \\
    &{}\hskip 2.2in               + \frac{z-x-\xi_z+\xi_x}3
\end{align}
and
\begin{align}\label{e:caseivb}
  T_{r-1}\(\frac{x+y+z+2}3\)
    &=   T_{r-1}\(\frac{x-\zet_x}3 + \frac{y-\zet_y}3
                                            + \frac{z-\zet_z}3\) \notag \\
    &\le T_{r-2}\(\frac{x-\zet_x}3\)+T_{r-2}\(\frac{y-\zet_y}3\)
                                   + T_{r-2}\(\frac{z-\zet_z}3\) \notag \\
    &{}\hskip 2.2in                + \frac{z-x-\zet_z+\zet_x}3 \notag \\
    &=   \frac13\,T_{r-1}(x-\zet_x) + \frac13\,T_{r-1}(y-\zet_y)
                                   + \frac13\, T_{r-1}(z-\zet_z) \notag \\
    &{}\hskip 2.2in               + \frac{z-x-\zet_z+\zet_x}3.
\end{align}
The result follows from \refe{caseiv}--\refe{caseivb}, \refe{Txizet}, and
\refe{xizet}.

\subsection*{Case (v): two of $x,y$, and $z$ are congruent to $1$
  modulo $3$, and the third is congruent to $2$ modulo $3$.}

In this case \refe{caseiv} remain valid, while \refe{caseiva} is to be
replaced with
\begin{multline*}
  T_{r-1}\(\frac{x+y+z-1}3\)
    \le \frac13\,T_{r-1}(x-\xi_x) + \frac13\,T_{r-1}(y-\xi_y)
                                   + \frac13\, T_{r-1}(z-\xi_z) \\
             + \frac{z-x-\xi_z+\xi_x+2}3,
\end{multline*}
and \refe{caseivb} with
\begin{multline*}
  T_{r-1}\(\frac{x+y+z+2}3\)
    \le \frac13\,T_{r-1}(x-\zet_x) + \frac13\,T_{r-1}(y-\zet_y)
                                   + \frac13\, T_{r-1}(z-\zet_z) \\
             + \frac{z-x-\zet_z+\zet_x+4}3.
\end{multline*}
The proof can now be completed as in Case (iv).
\end{proof}

\section{The isoperimetric problem: proofs of Theorems~\reft{lex} and
  \reft{isoper}}\label{s:isoper}

\begin{proof}[Proof of Theorem~\reft{lex}]
We assume that $m$ is fixed and use induction on $r$, for each $r$ proving
the equality
  $$ \prt_S(\cI_n(C_m^r))=m^r\ome_m(n/m^r) $$
for all $n\in[0,m^r]$. The case $r=1$ is easy in view of
$\ome_m(0)=\ome_m(1)=0$ and since $\ome_m(n/m)=1/m$ for $n=1\longc m-1$;
suppose, therefore, that $r\ge 2$.

Let $s_0$ be the smallest element of $S$. Denote by $H$ the subgroup of
$C_m^r$, generated by the set $S_0:=S\stm\{s_0\}$, and for brevity, write
$A:=\cI_n(C_m^r)$. For $i=0\longc m-1$, let $A_i:=A\cap(is_0+H)$ and
$n_i=|A_i|$. Notice, that if $n=tm+\rho$ with integer $t\ge 0 $ and
$\rho\in[1,m]$, then
\begin{equation}\label{e:u1urho}
  n_0\longe n_{\rho-1}=t+1\quad\text{and}\quad n_\rho\longe n_{m-1}=t.
\end{equation}
We have
  $$ \prt_S(A)
           = \prt_{S_0}(A_0)\longp \prt_{S_0}(A_{m-1}) + (n_0-n_{m-1}), $$
the first $m$ summands counting those pairs $(a,s)$ with $a\in A$ and
 $s\in S_0$ such that $a+s\notin A$, and the last summand counting pairs
$(a,s_0)$ with $a\in A$ such that $a+s_0\notin A$. By the induction
hypothesis, as applied to the subsets $A_i-is_0$ of the group $H$ with the
generating subset $S_0$, we have then
  $$ \prt_S(A)
         = m^{r-1}\ome_m\(\frac{n_0}{m^{r-1}}\)
                \longp m^{r-1}\ome_m\(\frac{n_{m-1}}{m^{r-1}}\)
                                                        + (n_0-n_{m-1}). $$
Now if $m$ divides $n$, then $n_0\longe n_{m-1}=n/m$ and the assertion
follows immediately. If, on the other hand, $m$ does not divide $n$, then in
view of \refe{u1urho} and by Lemma~\refl{takagi}, the right-hand side is
equal to
  $$ m^r\, \( \frac{\rho\, \ome_m((t+1)/m^{r-1})
      + (m-\rho)\, \ome_m(t/m^{r-1})}{m}\,+\,\frac1{m^r} \)
                     = m^{r} \ome_m(n/m^r), $$
completing the proof.
\end{proof}

To prove Theorem~\reft{isoper} we need the following simple lemma.
\begin{lemma}\label{l:absum}
For any integer $m\ge 2$ and real $x_1\longc x_m$, we have
  $$ |x_2-x_1|+|x_3-x_2|\longp |x_m-x_{m-1}|+|x_1-x_m|
                                       \ge 2\( \max_i x_i - \min_i x_i \). $$
\end{lemma}

\begin{proof}
Assume, without loss of generality, that $x_1$ is the smallest of the numbers
$x_1\longc x_m$, and let $j\in[1,m]$ be so chosen that $x_j$ is the largest
of these numbers. Then, by the triangle inequality,
\begin{multline*}
  |x_2-x_1|+|x_3-x_2|\longp |x_m-x_{m-1}|+|x_1-x_m| \\
        = |x_2-x_1| + |x_3-x_2| \longp |x_j-x_{j-1}| {\hskip 0.75in} \\
        {\hskip 0.5in} + |x_{j+1}-x_j| \longp |x_m-x_{m-1}| + |x_1-x_m| \\
        \ge |x_j-x_1| + |x_1-x_j| = 2(x_j-x_1).
\end{multline*}
\end{proof}

For further references, we record the following observation: if $m\ge 2$ and
$f\in\cF_m$, then, choosing in \refe{m} some of the numbers $x_i$ equal to
$0$, and the rest equal to $1$, in view of the boundary condition
\refe{bound-cond} we get
\begin{equation}\label{e:f(j/m)}
  f(n/m)\le 1;\ n=1\longc m-1.
\end{equation}

\begin{proof}[Proof of Theorem~\reft{isoper}]
We fix $m$ and use induction on $|G|$: assuming that the assertion is true
for any abelian group, the order of which is smaller than $|G|$ (and the
exponent of which divides $m$), we show that it is true for the group $G$.

Without loss of generality, we assume that $S$ is a minimal (under inclusion)
generating subset of $G$. Fix an element $s_0\in S$ and write
$S_0:=S\stm\{s_0\}$. If $S_0=\est$, then $G$ is cyclic of exponent $|G|$,
whence $|G|$ divides $m$ and therefore $f\in\cF_{|G|}$; consequently,
$f(|A|/|G|)\le 1$ by \refe{f(j/m)} and the assertion follows. Assuming now
that $S_0\ne\est$, let $H$ be the subgroup of $G$, generated by $S_0$; thus,
$H$ is proper and non-trivial. Let $l:=[G:H]$. Observe, that the exponent of
the quotient group $G/H$, and hence also its order $l$, are divisors of $m$,
and that $G/H$ is cyclic, generated by $s_0+H$. For $i=1\longc l$ set
$x_i:=|A\cap(is_0+H)|/|H|$.

Fix $i\in[1,l]$. By the induction hypothesis (as applied to the subset
$(A-is_0)\cap H$ of the group $H$ with the generating subset $S_0$), the
number of edges of $\Gam_S(G)$ from an element of $(is_0+H)\cap A$ to an
element of $(is_0+H)\stm A$ is at least $\frac1m\,|H|f(x_i)$. Furthermore,
the number of edges from $(is_0+H)\cap A$ to $((i+1)s_0+H)\stm A$ is at least
\begin{multline*}
  \max \{ |(is_0+H)\cap A|-|((i+1)s_0+H)\cap A| , 0 \} \\
    = |H| \max\{ x_i-x_{i+1}, 0 \}
                = \frac12\,|H|\big( |x_i-x_{i+1}| + x_i-x_{i+1} \big)
\end{multline*}
(where $x_{i+1}$ is to be replaced with $x_1$ for $i=l$). It follows that
\begin{multline}\label{e:loc1}
  \prt_S(A) \ge \frac1m\,|H|\,\big(f(x_1)\longp f(x_l)\big) \\
            + \frac12\, |H|\big(|x_1-x_2|\longp |x_{l-1}-x_l|+|x_l-x_1|\big).
\end{multline}
Choose $i,j\in[1,l]$ so that $x_i$ is the smallest, and $x_j$ is the largest
of the numbers $x_1\longc x_l$. By Lemma \refl{absum} and \refe{loc1} we have
then
\begin{align*}
  \prt_S(A)
      &\ge \frac1m\,|G|\,\frac{f(x_1)\longp f(x_l)}l  + |H|(x_j-x_i) \\
      &\ge \frac1m\,|G|\,\( \frac{f(x_1)\longp f(x_l)}l + (x_j-x_i) \).
\end{align*}
Recalling that $f\in\cF_m$ implies $f\in\cF_l$ in view of $l\mid m$, we get
\begin{align*}
  \prt_S(A)
      &\ge \frac1m\,|G|\,f\( \frac{x_1\longp x_l}l \)  \\
      &= \frac1m\,|G|\,f(|A|/|G|),
\end{align*}
as wanted.
\end{proof}

\section{The classes $\cF_m$: proofs of
  Propositions~\refp{exln1x} and~\refp{Fm-upper}, \\
  and Theorems~\reft{Fmp}, \reft{Fm}, and~\reft{funny}}\label{s:Fm}

Our proof of Proposition~\refp{exln1x} is based on the following lemma (which
the reader is recommended to compare with Theorem~\reft{funny}).
\begin{lemma}\label{l:concave-in-Fm}
Suppose that $f$ is a real-valued function, defined and concave on the
interval $[0,1]$. If the estimate
\begin{equation}\label{e:twostars1}
  f(\lam x_1+(1-\lam)x_2) \le \lam f(x_1)+(1-\lam)f(x_2) + (x_2-x_1)
\end{equation}
holds for all $\lam,x_1,x_2\in[0,1]$ with $x_1\le x_2$, then for any integer
$m\ge 2$ we have $f\in\cF_m$.
\end{lemma}

\begin{proof}
We fix integer $m\ge 2$ and real $x_1\longc x_m\in[0,1]$ with
 $\min_i x_i=x_1$ and $\max_i x_i=x_m$, and, assuming \refe{twostars1}, show
that \refe{m} holds true. For $i=1\longc m$ define $\lam_i\in[0,1]$ by
$x_i=\lam_i x_1+(1-\lam_i)x_m$ and let $\lam:=(\lam_1\longp\lam_m)/m$, so
that
  $$ f(x_i) \ge \lam_i f(x_1)+(1-\lam_i)f(x_m) $$
by concavity and, consequently,
\begin{equation}\label{e:fracge1}
  \frac{f(x_1)\longp f(x_m)}m \ge \lam f(x_1)+(1-\lam)f(x_m).
\end{equation}
On the other hand, we have
\begin{align}\label{e:ffracle1}
  f\( \frac{x_1\longp x_m}m \)
    &=   f(\lam x_1+(1-\lam)x_m) \notag \\
    &\le \lam f(x_1)+(1-\lam)f(x_m) + (x_m-x_1)
\end{align}
by \refe{twostars1}, and the result follows by comparing \refe{fracge1} and
\refe{ffracle1}.
\end{proof}

\begin{proof}[Proof of Proposition~\refp{exln1x}]
Since $f$ is concave on $[0,1]$, by Lemma~\refl{concave-in-Fm} it suffices to
prove \refe{twostars1}. The case $\lam\in\{0,1\}$ is trivial, and we assume
below that $0<\lam<1$. Denote by $\Del_\lam(x_1,x_2)$ the difference of the
left-hand side and the right-hand side of \refe{twostars1}. Since the second
partial derivative of $\Del_\lam(x_1,x_2)$ with respect to $x_2$ is
  $$ \frac{\lam(1-\lam)ex_1}{(\lam x_1+(1-\lam)x_2)x_2} \ge 0,
                                                        \quad x_2\in(0,1), $$
the largest value of $\Del_\lam(x_1,x_2)$ for any fixed $\lam$ and $x_1$ is
attained either for $x_2=x_1$, or for $x_2=1$; consequently, we can confine
to these two cases. Indeed, \refe{twostars1} holds true in a trivial way for
$x_2=x_1$, and we therefore assume that $x_2=1$; thus, it remains to prove
that
  $$ \Del_\lam(x_1,1)=f(\lam x_1+1-\lam)-\lam f(x_1)-1+x_1\le 0,
                                                        \ x_1\in[0,1]. $$
To this end we just observe that the second derivative of $\Del_\lam(x_1,1)$
with respect to $x_1$ is
  $$ \frac{\lam(1-\lam)e}{(\lam x_1+1-\lam)x_1} > 0, \quad x_1\in(0,1), $$
and that $\Del_\lam(0,1)=f(1-\lam)-1\le 0$ and $\Del_\lam(1,1)=0$.
\end{proof}

We now turn to the proof of Theorem~\reft{Fmp}.
\begin{lemma}\label{l:continuity}
For every integer $m\ge 2$, all functions from the class $\cF_m$ are
continuous on $(0,1)$.
\end{lemma}

\begin{proof}
We fix an integer $m\ge 2$, a function $f\in\cF_m$, and a number
$x_0\in(0,1)$, and show that $f$ is continuous at $x_0$. Let
 $l:=\min\{\liminf_{x\to x_0} f(x), f(x_0)\}$ and
 $L:=\max\{\limsup_{x\to x_0} f(x), f(x_0)\}$. It suffices to prove that
$l\ge L$. For this, choose two sequences $\{\xi_k\}_{k=1}^\infty$ and
$\{\zeta_k\}_{k=1}^\infty$ with all terms in $(0,1)$, converging to $x_0$,
and satisfying $f(\xi_k)\to l$ and $f(\zeta_k)\to L$. In addition, we request
$m\zeta_k-(m-1)\xi_k\in(0,1)$ to hold for any integer $k\ge 1$; in view of
$m\zeta_k-(m-1)\xi_k\to x_0$, this can be arranged simply by dropping a
finite number of terms from each sequence. By \refe{m} we have then
  $$ f(\zeta_k)
       \le \frac{(m-1)f(\xi_k)+f(m\zeta_k-(m-1)\xi_k)}m + o(1) $$
as $k\to\infty$, and it remains to observe that the left-hand side is
$L+o(1)$, while the right-hand side is at most $((m-1)l+L)/m+o(1)$.
\end{proof}

We remark that the functions from the classes $\cF_m$ are not necessarily
continuous at the endpoints of the interval $[0,1]$. Indeed, for any
$f\in\cF_m$ and $a>0$, letting
  $$ f_a(x) = \begin{cases}
                     f(x), &\text{if}\ x\in\{0,1\}, \\
                     f(x)-a &\text{if}\ x\in(0,1),
                   \end{cases} $$
we have $f_a\in\cF_m$, and either $f$ or $f_a$ is discontinuous at $0$ and
$1$. However, a slight modification of the proof of Lemma \refl{continuity}
shows that the potential discontinuities of a function $f\in\cF_m$ at the
endpoints of $[0,1]$ are removable; that is, the limits $\lim_{x\to0+} f(x)$
and $\lim_{x\to1-} f(x)$ exist and are finite.

The following corollary follows readily from Theorem~\reft{Fm} and the
estimate \refe{max-omega}. However, since we have not proved
Theorem~\reft{Fm} yet, we use here an independent argument.
\begin{corollary}\label{c:mm-1}
For any integer $m\ge 2$ and any function $f\in\cF_m$, we have
 $\sup f\le
m/(m-1)$.
\end{corollary}

\begin{proof}
By Lemma \refl{continuity}, it suffices to show that $f(n/m^r)\le m/(m-1)$
holds for all integer $r\ge 0$ and $n\in[0,m^r]$. Indeed, using induction on
$r$, we prove the slightly stronger estimate
  $$ f(n/m^r) \le 1+1/m \longp 1/m^{r-1}. $$
For $r=0$ this reduces to the boundary condition \refe{bound-cond}. Assuming
that $r\ge 1$ and $n$ is not divisible by $m$, write $n=tm+\rho$ with integer
$t$ and $\rho\in[1,m-1]$. Then
  $$ \frac n{m^r} = \frac{(m-\rho)(t/m^{r-1})+\rho\,((t+1)/m^{r-1})}m $$
so that by \refe{m} and the induction hypothesis,
\begin{align*}
  f\(\frac{n}{m^r}\)
    &\le \(1-\frac \rho m\)f\(\frac{t}{m^{r-1}}\)
               + \frac \rho m\, f\(\frac{t+1}{m^{r-1}}\) + \frac1{m^{r-1}} \\
    &\le \(1-\frac \rho m\) \(1+\frac1m \longp \frac1{m^{r-2}} \) \\
    &\phantom{ (1-\ ) }
           + \frac \rho m\, \(1+\frac1m \longp \frac1{m^{r-2}} \)
                                                        +  \frac1{m^{r-1}} \\
    &= 1+\frac1m\longp \frac1{m^{r-1}}.
\end{align*}
\end{proof}

\begin{proof}[Proof of Theorem~\reft{Fmp}.]
With Corollary \refc{mm-1} in mind, we set
  $$ F_m(x) := \sup \{ f(x)\colon f\in\cF_m \},\quad x\in[0,1]. $$
In view of Proposition~\refp{exln1x}, we have $F_m(0)\ge 0$, $F_m(1)\ge 0$,
and $F_m(x)>0$ for $x\in(0,1)$; indeed, $F_m(0)=F_m(1)=0$ by
\refe{bound-cond}. We now show that
\begin{equation}\label{e:FmincFm}
  F_m\in\cF_m;
\end{equation}
this will immediately imply continuity of $F_m$ on $(0,1)$ (by
Lemma~\refl{continuity}) and show that $F_m(x)=F_m(1-x)$ (since if $f$
belongs to $\cF_m$, then so does the function $x\mapsto f(1-x)$).

To prove \refe{FmincFm} we notice that, given $\eps>0$ and
 $x_1\longc x_m\in[0,1]$ with $\min_i x_i=x_1$ and $\max_i x_i=x_m$, we can find
$f\in\cF_m$ such that
  $$ F_m\(\frac{x_1\longp x_m}m\) \le f\(\frac{x_1\longp x_m}m\)+\eps, $$
and then, by \refe{m},
\begin{align*}
  F_m\(\frac{x_1\longp x_m}m\)
     &\le \frac{f(x_1)\longp f(x_m)}m\ + (x_m-x_1) + \eps \\
     &\le \frac{F_m(x_1)\longp F_m(x_m)}m\ + (x_m-x_1) + \eps.
\end{align*}
Taking the limits as $\eps\to 0$ gives
  $$ F_m\(\frac{x_1\longp x_m}m\)
                     \le \frac{F_m(x_1)\longp F_m(x_m)}m\ + (x_m-x_1), $$
whence $F_m\in\cF_m$.

To complete the proof it remains to show that $F_m$ is continuous at the
endpoints of the interval $[0,1]$. As remarked above, a slight modification
of the proof of Lemma~\refl{continuity} shows, in view of \refe{FmincFm},
that the limits
 $\lim_{x\to 0+} F_m(x)$ and $\lim_{x\to 1-} F_m(x)$ exist and are finite.
Moreover, from $F_m(x)=F_m(1-x)$ it follows that these limits are equal to
the same number $L$, and we want to show that $L=0$. Since $F_m$ is positive
on $(0,1)$, we have $L\ge 0$. To show, on the other hand, that $L\le 0$, we
observe that if $\{\xi_k\}_{k=1}^\infty$ is a sequence satisfying $\xi_k\to
0$ and $\xi_k\in(0,1/m]$ for any $k\ge 1$, then, by \refe{FmincFm} and
\refe{m},
  $$ L+o(1) = F_m(\xi_k) \le \frac{(m-1)F_m(0)+F_m(m\xi_k)}m + o(1)
                                                     = \frac1m\,L + o(1) $$
as $k\to\infty$.
\end{proof}

\begin{proof}[Proof of Theorem~\reft{funny}]
Considering $x<y$ fixed, let
  $$ f_{x,y}(\lam) := \frac1{y-x}
            \, \big( f(\lam x+(1-\lam)y)-\lam f(x)-(1-\lam)f(y) \big). $$
Fix arbitrarily $\lam_1\longc\lam_m\in[0,1]$ with $\min_i\lam_i=\lam_1$ and
$\max_i\lam_i=\lam_m$, and write $x_i:=\lam_i x+(1-\lam_i)y;\ i\in[1,m]$.
Notice that $\min_i x_i=x_m$ and $\max_i x_i=x_1$, and if $x,y\in[0,1]$, then
also $x_1\longc x_m\in[0,1]$. Hence, by \refe{m},
\begin{align*}
  (y-x) &f_{x,y}\(\frac{\lam_1\longp\lam_m}{m}\) \\
    &=   f\( \frac{x_1\longp x_m}m \)
                  - \frac1m\,\sum_{i=1}^m \( \lam_if(x)+(1-\lam_i)f(y) \) \\
    &\le \frac1m\,\sum_{i=1}^m
             \big( f(x_i) - \lam_if(x) - (1-\lam_i)f(y) \big) + (x_1-x_m) \\
    &=   \frac{y-x}m\, \sum_{i=1}^m f_{x,y}(\lam_i) + (y-x)(\lam_m-\lam_1).
\end{align*}
This shows that $f_{x,y}\in\cF_m$. Consequently, $f_{x,y}(\lam)\le F_m(\lam)$
by the extremal property of the function $F_m$ (cf.~Theorem~\reft{Fmp}) and
the assertion follows.
\end{proof}

\begin{proof}[Proof of Theorem~\reft{Fm}]
By continuity of the functions $\ome_m$ and $F_m$ (see Theorem~\reft{Fmp}),
to show that $F_m\le m\ome_m$ it suffices to prove that for any integer $r\ge
0$ and $n\in[0,m^r]$, we have $F_m(n/m^r)\le m\ome_m(n/m^r)$. We use
induction on $r$, and for each $r$ prove the assertion for all $n\in[0,m^r]$.

The case $r=0$ is immediate from $F_m(0)=0=m\ome_m(0)$ and
$F_m(1)=0=m\ome_m(1)$. For $r\ge 1$ we assume, without loss of generality,
that $n$ is not divisible by $m$, and write $n=mt+\rho$ with integer $t$ and
$\rho\in[1,m-1]$. From $F_m\in\cF_m$, the induction hypothesis, and
Lemma~\refl{takagi} we have then
\begin{align*}
  F_m(n/m^r) &\le \frac{(m-\rho)\,F_m(t/m^{r-1})
                       + \rho\,F_m((t+1)/m^{r-1})}m \, + \, \frac1{m^{r-1}} \\
           &\le (m-\rho)\,\ome_m\(\frac{t}{m^{r-1}}\)
                       + \rho\,\ome_m\(\frac{t+1}{m^{r-1}}\) \, + \, \frac1{m^{r-1}} \\
           &=   m\ome_m\(\frac{n}{m^r}\),
\end{align*}
as wanted.

Next, we prove that $F_m=m\ome_m$ for $m\in\{2,3,4\}$. The case $m=2$ follows
from the estimate $F_2\le 2\ome_2$ which we have just obtained and
Boros-P\'ales inequality \refe{bp}, showing that $2\ome_2\in\cF_2$ and,
therefore, $F_2\ge 2\ome_2$. Similarly, the case $m=3$ follows from
 $F_3\le 3\ome_3$ and Theorem~\reft{bp3} showing that
$3\ome_3\in\cF$. For the case $m=4$ we notice that, in view of $\cF_4=\cF_2$
and \refe{omega4=omega2},
  $$ F_4 = F_2 = 2\ome_2 = 4\ome_4. $$

It remains to show that $F_m\ne m\ome_m$ for $m\ge 5$. To this end we observe
that in this case $4/m^2\le 1/m\le 4/m\le 1-1/m$, whence
  $$ m\ome_m(4/m^2) = m\|4/m^2\|_{1/m} + \|4/m\|_{1/m} = 5/m, $$
whereas, by \refe{f(j/m)}, \refe{m}, and $F_m(0)=0$,
  $$ F_m\(\frac4{m^2}\) \le \frac{(m-2)F_m(0)+2F_m(2/m)}m + \frac2m
                                                             \le \frac4m. $$
\end{proof}

In connection with Theorem~\reft{Fm} we remark that the estimate
 $F_m\le m\ome_m$ and the inequality $F_m\ne m\ome_m$ for $m\ge 5$ also follow
from Theorems~\reft{lex} and \reft{isoper}, the latter of them applied with
$f=F_m$, and the well-known and easy-to-verify fact that the sets
$A=\cI_n(C_m^r)$ do \emph{not} minimize the quantity $\prt_S(A)$ for $m\ge
5$. This is yet another indication of the intrinsic relation between the
discrete isoperimetric problem and the functions $\ome_m$ and $F_m$.

Finally, we prove Proposition~\refp{Fm-upper}.
\begin{proof}[Proof of Proposition~\refp{Fm-upper}]
Suppose that $f$ is a real-valued function, defined on the interval $[0,1]$
and satisfying the boundary condition \refe{bound-cond} and the inequality
\refe{twostars1}
for all $\lam,x_1,x_2\in[0,1]$ with $x_1\le x_2$. For real $\xi\in[0,1]$ and
integer $k\ge 1$, applying \refe{twostars1} with $x_1=0,\ x_2=\xi^{k-1}$, and
$\lam=1-\xi$, we obtain
  $$ f(\xi^k)=f(\lam x_1+(1-\lam)x_2)\le (1-\lam)f(x_2) + x_2
                                             = \xi^{k-1}+\xi f(\xi^{k-1}); $$
iterating,
\begin{equation}\label{e:kx^k-1}
  f(\xi^k) \le 2\xi^{k-1} + \xi^2f(\xi^{k-2}) \le \cdots \le k\xi^{k-1}.
\end{equation}
For $x\in(0,1)$, we use the resulting estimate with $k:=\lcl\ln(1/x)\rcl$ and
$\xi:=x^{1/k}$ to get
  $$ f(x) < (1+\ln(1/x))\,x\cdot x^{-1/k} \le ex\ln(e/x). $$
To complete the proof it remains to observe that, by Corollary \refc{mm-1},
we have $F_m\le m/(m-1)$, and therefore Theorem~\reft{funny} shows that the
function $f=(1-m^{-1})F_m$ satisfies \refe{twostars1}.
\end{proof}

\section*{Appendix: proof of inequalities
  \refe{ome2-bounds}--\refe{omem-bounds}.}

We prove here inequalities \refe{ome2-bounds}, \refe{ome3-bounds}, and
\refe{omem-bounds}; inequality \refe{ome4-bounds} is immediate from
\refe{ome2-bounds} and \refe{omega4=omega2}. The proofs use the identities
\begin{equation}\label{e:no-time-to-think}
  \ome_m(x) = \|x\|_{1/m} + \frac1m\,\|mx\|_{1/m} \longp
              \frac1{m^k}\,\|m^k x\|_{1/m} + \frac1{m^{k+1}}\,\ome_m(m^{k+1}x)
\end{equation}
and
\begin{equation}\label{e:par-per}
  \ome_m(n\pm x)=\ome_m(x),
\end{equation}
valid for any integer $m\ge 2$, $k\ge 0$, and $n$, and any choice of the
sign.

\begin{proof}[Proof of the inequality \refe{ome2-bounds}]
As an immediate corollary of \refe{no-time-to-think}, for each $x\in[0,1/2]$
we have $\ome_2(x)=x+\frac12\,\ome_2(2x)$. On the other hand, for any fixed
$C>0$, the function $f_C(x):=x\log_2(C/x)$ satisfies the very same functional
equation: $f_C(x)=x+\frac12\,f_C(2x)$. Hence,
  $$ f_C(x)-\ome_2(x) = \frac12\,\big(f_C(2x)-\ome_2(2x)\big),\ x\in(0,1/2], $$
showing that it suffices to prove the estimates in question in the range
$x\in[1/2,1]$. To establish the lower bound we now observe that
  $$  \ome_2(x) \ge \|x\|+\frac12\,\|2x\|
          =   \begin{cases}
                 1/2 \quad &\text{if}\quad 1/2 \le x \le 3/4, \\
                 2-2x   \quad &\text{if}\quad 3/4 \le x \le 1,
               \end{cases} $$
and using some basic calculus, one verifies easily that the function in the
right-hand side is at least as large as $x\log_2(1/x)$ for all $x\in[1/2,1]$.

Turning to the upper bound, we notice that the function $f_{4/3}$ is
decreasing on the interval $[1/2,\infty)$, and that the largest value
attained by the Takagi function is known to be $\max \ome_2=2/3$ (see
\refb{la} or \refb{ak}). As a result,
  $$ \ome_2(x) \le 2/3 = f_{4/3}(2/3) \le f_{4/3}(x),\ x\in[1/2,2/3], $$
and it remains to prove the upper bound in the case $x\in(2/3,1]$. To this
end, for integer $r\ge 0$ we let
  $$ b_r := 1-\frac14-\frac1{4^2}\longm\frac1{4^r}, $$
and use induction on $r$ to show that $\ome_2(x)\le f_{4/3}(x)$ for all
$x\in[b_{r+1},b_r]$. If $r=0$, then $x\in[3/4,1]$; in this range we have
\begin{align*}
  \ome_2(x)
    &= \|x\| + \frac12\,\|2x\| + \frac14\,\|4x\| + \frac18\,\ome_2(8x) \\
    &\le (1-x) + \frac12\,(2-2x) + \frac14\,\|4x\| + \frac1{12} \\
    &= \frac{25}{12} - 2x + \frac14\,\|4x\|,
\end{align*}
and a simple verification confirms that the expression in the right-hand side
is smaller than $f_{4/3}(x)$ for $x\in[3/4,1]$. Finally, assuming
$x\in[b_{r+1},b_r]$ with $r\ge 1$, we observe that this implies
$4x-2\in[b_r,b_{r-1}]$; hence, using \refe{par-per}, the induction
hypothesis, and the fact that $x\in[b_{r+1},b_r]\seq(2/3,3/4]$, we get
\begin{align*}
  \ome_2(x)
    &=   \|x\| + \frac12\,\|2x\| + \frac14\,\ome_2(4x-2) \\
    &\le (1-x) + \frac12\,(2x-1) + \frac14\,(4x-2)\log_2\frac4{3(4x-2)} \\
    &=   \frac12 + \(x-\frac12\) \log_2\frac4{3(4x-2)} \\
    &\le x\log_2\frac4{3x},
\end{align*}
the last inequality following by observing that both sides are equal at
$x=2/3$ and comparing the derivatives. This completes the proof.
\end{proof}

\begin{proof}[Proof of the inequality \refe{ome3-bounds}]
Similarly to the proof of \refe{ome2-bounds}, writing $f_C(x):=x\log_3(C/x)$,
for every $x\in(0,1/3]$ we have $\ome_3(x)=x+\frac13\,\ome_3(3x)$ and also
$f_C(x)=x+\frac13\,f_C(3x)$. Hence,
  $$ f_C(x)-\ome_3(x) = \frac13\,\big(f_C(3x)-\ome_3(3x)\big),\ x\in(0,1/3], $$
showing that we can assume $x\in [1/3,1]$.

Observing that if $x\in[1/3,4/9]$, then
  $$ \ome_3(x) \ge \|x\|_{1/3} + \frac13\,\|3x\|_{1/3}
                        = \frac13 + \frac13\,(3x-1) = x \ge x\log_3(1/x), $$
and if $x\in[4/9,1]$, then
  $$ \ome_3(x) \ge \|x\|_{1/3} \ge x\log_3(1/x) $$
(straightforward verification is left to the reader), we get the lower bound.

For the upper bound, we can further restrict the range to consider from
$[1/3,1]$ to $[1/3,1/2]$: for once the estimate is established in this
narrower range, it readily extends onto the interval $[1/2,2/3]$ in view of
  $$ \ome_3(x) = \ome_3(1-x),
                        \ f_{3/2}(x) \le f_{3/2}(1-x),\ 0<x\le 1/2, $$
and onto the interval $[2/3,1]$ since for any $x$ in this interval, by
\refe{no-time-to-think} and \refe{par-per} we have
  $$ \ome_3\(\frac{2-x}{3}\) = \frac13 + \frac13\,\ome_3(x),  $$
whence (assuming the upper bound is proved in $[1/3,1/2]$)
\begin{multline*}
  \ome_3(x) = 3 \ome_3\(\frac{2-x}{3}\) - 1
              \le (2-x)\log_3\frac9{2(2-x)} - 1 \le x\log_3\frac3{2x},
                                                        \ x\in[2/3,1].
\end{multline*}
(For the last inequality observe that both sides are equal for $x=1$, and
compare the derivatives).

Thus, it remains to prove the upper bound for $x\in[1/3,1/2]$. To this end,
for integer $r\ge 1$ we let
  $$ b_r := \frac13\longp\frac1{3^r}, $$
and use induction on $r$ to show that $\ome_3(x)\le f_{3/2}(x)$ for all
$x\in[b_r,b_{r+1}]$. If $r=1$, then $x\in[1/3,4/9]$; in view of
\refe{max-omega}, in this range we have
\begin{align*}
  \ome_3(x)
    &= \|x\|_{1/3} + \frac13\,\|3x\|_{1/3} + \frac19\,\|9x\|_{1/3}
                                          + \frac1{27}\,\ome_3(27x) \\
    &\le \frac13 + \frac13\,(3x-1)
         + \frac19\,\min\left\{\frac13, 4-9x\right\} + \frac1{54} \\
    &=   \min\left\{x+\frac1{18},\frac{25}{54} \right\},
\end{align*}
and a simple verification confirms that the expression in the right-hand side
is smaller than $f_{3/2}(x)$ for $x\in[1/3,4/9]$. Assuming now that $r\ge 2$,
we observe that $x\in[b_r,b_{r+1}]$ implies $3x-1\in[b_{r-1},b_r]$; hence, by
the induction hypothesis, for all $x$ in this range we have
\begin{align*}
  \ome_3(x)
    &=   \|x\|_{1/3} + \frac13\,\ome_3(3x-1) \\
    &\le \frac13 + \frac13\,(3x-1)\log_3\frac3{2(3x-1)} \\
    &\le x\log_3\frac3{2x}.
\end{align*}
(For the last inequality compare the values of both sides at $1/2$ and their
derivatives for $1/3<x<1/2$). This completes the proof.
\end{proof}

\begin{proof}[Proof of the inequality \refe{omem-bounds}]
As in the proofs of \refe{ome2-bounds} and \refe{ome3-bounds}, we can confine
to the range $x\in[1/m,1]$ where the upper bound readily follows from
\refe{max-omega}:
  $$ \ome_m(x)\le \frac1{m-1} \le x\log_m(3/2x), \ x\in[1/m,1]. $$
(Notice that the right-hand side is a concave function, hence attains its
minimum at an endpoint.) For the lower bound we observe that the function
$x\log_m(e/mx)$ is decreasing for $x\ge 1/m$, whence
  $$ \ome_m(x) \ge \frac1m \ge x\log_m(e/mx), \quad x\in[1/m,1-1/m] $$
and
  $$ \ome_m(x) \ge 0 > x\log_m(e/mx), \quad x\in[1-1/m,1]. $$
\end{proof}

\section*{Acknowledgement}
The author is grateful to Mikhail Muzychuck for his interest and a number of
important remarks, and to Sergei Bezrukov for generously sharing his
expertise in the discrete isoperimetric problem.

\bigskip
\end{document}